\documentclass{article}

\usepackage{amsmath,amsfonts,amssymb,amsthm,tikz}
\usepackage{algorithm}
\usepackage{algorithmic}
\usepackage[utf8]{inputenc}
\usepackage{cite}

\usepackage{hyperref}
\hypersetup{hypertex=true,
            colorlinks=true,
            linkcolor=blue,
            anchorcolor=blue,
            citecolor=blue}

\newcommand{\norm}{{\mathcal{N}}}

\newcommand{\C}{{\mathbb{C}}}

\newcommand{\R}{{\mathbb{R}}}
\newcommand{\Q}{{\mathbb{Q}}}
\newcommand{\Z}{{\mathbb{Z}}}
\newcommand{\N}{{\mathbb{N}}}

\newcommand{\height}{\mathrm{h}}

\newcommand{\gerp}{{\mathfrak{p}}}
\newcommand{\gerq}{\mathfrak{q}}

\newcommand{\eps}{\varepsilon}
\newcommand{\ph}{\varphi}

\newcommand\tho{{\text{th}}}

    \makeatletter
    \let\@fnsymbol\@alph
    \makeatother

\title{On big primitive divisors of Fibonacci numbers}
\author{Haojie Hong}
\date{Version of \today}

\newtheorem{theorem}{Theorem}[section]
\newtheorem{proposition}[theorem]{Proposition}
\newtheorem{lemma}[theorem]{Lemma}

\newtheorem{remark}[theorem]{Remark}

\numberwithin{equation}{section}

\makeatletter

\renewcommand*\l@section[2]{%
  \ifnum \c@tocdepth >\z@
    \addpenalty\@secpenalty
    \addvspace{0.2em \@plus\p@}%
    \setlength\@tempdima{1.5em}%
    \begingroup
      \parindent \z@ \rightskip \@pnumwidth
      \parfillskip -\@pnumwidth
      \leavevmode \bfseries
      \advance\leftskip\@tempdima
      \hskip -\leftskip
      #1\nobreak\hfil \nobreak\hb@xt@\@pnumwidth{\hss #2}\par
    \endgroup
  \fi}

\makeatother

\allowdisplaybreaks[1]

\begin{document}

\hfuzz 4.3pt

\maketitle

\renewcommand\thefootnote{}

\begin{abstract}
In this note, we prove that for any given positive integer $\kappa$,
when $n$ is bigger than a constant explicitly depending on $\kappa$, the $n$-th Fibonacci number has a primitive divisor not less than $(\kappa+1) n-1$.
\end{abstract}

{\footnotesize

\tableofcontents

}

\section{Introduction}
Let $\alpha$ and $\beta$ be complex numbers such that $\alpha+\beta$ and $\alpha \beta$ are non-zero coprime integers and $\alpha/\beta$ is not a root of unity. For every $n\in\N$, let \[
u_n = \frac{\alpha^n - \beta^n}{\alpha - \beta}.
\]
One can verify that all $u_n$ are integers and they are called \emph{Lucas numbers}. In particular, if we set $\alpha=(1+\sqrt{5})/2$ and $\beta=(1-\sqrt{5})/2$, the Lucas numbers $(u_n)_{n\ge 0}$ are just \emph{Fibonacci numbers} \[
1,1,2,3,5,8,13,21,34,55,89, 144, 233\cdots,
\] which we denote by $F_n$ in the rest of this note.

A prime number $p$ is called a \emph{primitive divisor} of a Lucas number $u_n$ if $p\mid u_n$ but $p\nmid (\alpha-\beta)^2 u_1\cdots u_{n-1}$. Note that a primitive divisor $p $ of a Lucas number $u_n$ satisfies $p\equiv\pm 1\pmod n$. 
In 1892 Zsigmondy\cite{Zsi1892} proved that $u_n$ has a primitive divisor for $n>6$ when $\alpha, \beta\in\Z$. Later in 1913 Carmichael \cite{Car1913} proved that $u_n$ has a primitive divisor for $n>12$ if $\alpha, \beta\in\R$. According to Carmichael's result, it is clear that $F_n$ has a primitive divisor for $n\ge 7$ and $n\neq 12$.

In this note, we prove that when $n$ is big enough, $F_n$ always has a primitive divisor different from $n\pm 1, 2n\pm 1,\cdots, \kappa n\pm 1$ for fixed $\kappa$. The proofs use an adaptation of Stewart's argument\cite{St13, BHS22}, which use Yu's estimates\cite{Yu13} for $p$-adic logarithmic forms.

Our main results are Theorem \ref{thm_smallkappa} and Theorem \ref{thm_bigkappa}.

\begin{theorem}\label{thm_smallkappa}
Let $\kappa$ be a fixed positive integer and assume $\kappa\le 10^6$. There exists a positive integer $n_0$ depending on $\kappa$, which can be explicitly computed,
 such that for any positive integer $n\ge \exp(n_0)$, the Fibonacci number $F_n$ has a primitive divisor distinct from each of the $2\kappa$ numbers\[
n\pm 1, 2n\pm 1,\cdots, \kappa n\pm 1.
\]
In particular, $F_n$ has a prime divisor greater than or equal to $(\kappa+1)n-1$, when $n\ge \exp(n_0)$.
\end{theorem}

Actually, Stewart's work\cite[Lemma 4.3]{St13} holds particularly for primitive divisors of Fibonacci numbers.  This combined with estimate of the lower bound implies the existence of $n_0$ in the above theorem. However, Stewart does not give an explicit expression for $n_0$. Bilu et al.\cite[Theorem 1.3]{BHS22} gives an explicit expression for such $n_0$, but it is huge numerically. This note is aimed to obtain an explicit and numerically sharp expression for $n_0$.

Suitable values of $n_0$ for certain $\kappa$ can be found in the following Tables 1 and 2. In fact, a SageMath script (see Appendix A1) was made so that we are able to compute value of $n_0$ for any given $1\le \kappa \le 10^6$.

\begin{table}[htbp]
\centering
\begin{tabular}{|c|c|c|c|c|c|c|c|c|c|c|}
\hline
$\kappa$ & $1$ & $2$ & $3$  & $4$ & $5$ &  $6$ &  $7$  & $8$ & $9$ & $10$ \\ \hline
 $n_0$ &  $7607$ & $8006$ & $8257$ & $8443$ & $8588$ & $8710$ & $8815$ & $8904$ &  $8984$ &  $9057$\\ \hline
\end{tabular}\caption{optimal values of $n_0$ for $\kappa = 1,2,\ldots,10$}
\end{table}
\begin{table}[htbp]
\centering
\begin{tabular}{|c|c|c|c|c|c|c|c|c|c|}
\hline
$\kappa$ & $20$ & $30$ & $40$  & $50$ & $100$ &  $1000$ &  $10000$  & $10^5$ & $10^6$ \\ \hline
 $n_0$ &  $9544$ & $9831$ & $10036$ & $10196$ & $10701$ & $12405$ & $14121$ & $15841$ &  $17575$  \\ \hline
\end{tabular}\caption{optimal values of $n_0$ for some other $\kappa $}
\end{table}

\begin{remark}
For larger $\kappa$, say $10^6<\kappa< \exp(250000)$,
although the computation of the lower bound of $n$ may be very slow, it is still possible. In fact, one can just choose $k$ to be small in inequality \eqref{emain} to get an estimate for arbitrary $\kappa$,  even though it may not be optimal.  The next theorem characterizes the situation when $\kappa$ is particularly large.
\end{remark}

\begin{theorem}\label{thm_bigkappa}
Let $\kappa$ be a  positive integer and assume that $\log\kappa\ge 250000$. Then for any positive integer \[
n\ge \exp(143\log\kappa\log\log\kappa),
\]
the Fibonacci number $F_n$ has a primitive divisor distinct from each of the following $2\kappa$ numbers \[
n\pm 1, 2n\pm 1,\cdots, \kappa n\pm 1.
\]
In particular, $F_n$ has a prime divisor greater than or equal to $(\kappa+1)n-1$, when $n\ge  \exp(143\log\kappa\log\log\kappa)$.
\end{theorem}

\section{Notation and preliminaries}
Let $K$ be the filed $\Q(\sqrt{5})$. We denote by~$\sigma$  the non-trivial automorphism of~$K$ and by $\norm$ the $K/\Q$-norm. 

We denote 
$$
\eta:=\frac{1+\sqrt5}{2},\qquad  \gamma:=\frac\eta{\eta^\sigma}=-\eta^2. 
$$
Then
$$
F_n= \frac{\eta^n-(\eta^\sigma)^n}{\sqrt5}=\frac{(\eta^\sigma)^n}{\sqrt5}(\gamma^n-1). 
$$
Note that~$\eta$ is the fundamental unit of~$K$. 

Let ${p\ne 5}$ be a prime number, and~$\gerp$ be a prime of~$K$ above~$p$. We denote by ${f=f_p}$ the residual degree of~$\gerp$ (it depends only on~$p$, not on~$\gerp$). Note that~$p$ splits in~$K$ if ${p\equiv \pm1\pmod 5}$ and is inert if ${p\equiv \pm2\pmod5}$. Hence
\begin{equation}
\label{efp}
f_p=
\begin{cases}
1, &\text{if ${p\equiv \pm1\pmod 5}$},\\
2, &\text{if ${p\equiv \pm2\pmod 5}$}. 
\end{cases}
\end{equation}

We denote by $\height(x)$ the usual absolute logarithmic height of $x\in K$: \[
\height(x) = [K: \Q]^{-1}\sum_{v\in M_K}[K_v: \Q_v]\log^+|x|_v,
\]
where $\log^+(x):= \max\{\log(x), 0\}$.
We have 
\begin{equation}
\label{ehega}
\height(\gamma)=\log\eta. 
\end{equation}
Note also the lower bound 
\begin{equation}
\label{ehegeta}
\height(\alpha)\ge \frac12\log \eta \qquad (\alpha \in K^\times \smallsetminus\{1,-1\}). 
\end{equation}
Indeed, if~$\alpha$ is not a Dirichlet unit, then for some prime number~$p$ we have 
${\height(\alpha) \ge (1/2)\log p}$, even better than~\eqref{ehegeta}.
And if~$\alpha$ is a unit, then~\eqref{ehegeta} clearly holds, because~$\eta$ is the fundamental unit. 

We also need the following elementary lemma.

\begin{lemma}
\label{lna}
Let $\eps, A$  be real numbers satisfying ${0<\eps\le 1}$ and $ {A\ge e^{2\eps}\eps^{-\eps}} $. Let ${x> 0}$ satisfy 
 ${x\le A(\log x)^\eps}$. Then 
\begin{equation}
\label{ealoga}
x<  A(2\log A)^\eps
\end{equation}

\end{lemma}

\begin{proof}
Assume that ${\eps=1}$. Then we have ${A\ge e^{2}}$ and ${x\le A\log x}$. For ${A\ge e^{2}}$ we have ${2A\log A<A^2}$. 
Hence, assuming that ${x\ge 2A\log A}$ we obtain ${x/\log x >A}$, a contradiction. This proves the result for ${\eps=1}$. The general case reduces to the case just proved: rewriting the hypothesis as 
$
{x^{1/\eps}\le A^{1/\eps}\eps\log(x^{1/\eps})} 
$  and noting that ${A^{1/\eps}\eps \ge e^2}$, 
we obtain 
$$
x^{1/\eps}\le 2A^{1/\eps}\eps\log(A^{1/\eps}\eps) = 2A^{1/\eps}(\log A+ \eps \log\eps) \le 2A^{1/\eps}\log A, 
$$ 
which is~\eqref{ealoga}. 
\end{proof}

\section{The Theorem of Yu}

The following is an adaptation of (a special case of) the Main Theorem from~\cite{Yu13}. 

\begin{theorem}
\label{thyu}
Let ${k\ge 8}$ be an integer and~$\gerp$ a prime of~$K$ with underlying rational prime~$p$ satisfying 
\begin{equation}
\label{etechnical}
p\ge e^{3k}k^3. 
\end{equation} 
Let ${\alpha_1, \ldots, \alpha_k\in K^\times}$ be multiplicatively independent $\gerp$-adic units,  satisfying the Kummer condition 
\begin{equation}
\label{ekummer}
\bigl[K(\sqrt{\alpha_1}, \ldots, \sqrt{\alpha_k}):K\bigr]=2^k.
\end{equation}
Assume also that they are all of norm~$1$:
\begin{equation}
\label{enormone}
\norm\alpha_j=1 \qquad(j=1, \ldots, k). 
\end{equation}
The for  ${b_1, \ldots, b_k \in \Z\smallsetminus\{0\}}$  we have 
\begin{equation}
\label{eyu}
\nu_\gerp\bigl(\alpha_1^{b_1}\cdots \alpha_k^{b_k}-1\bigr) \le C_0 \max\{\log B, (k+1)f_p\log p\}\height(\alpha_1)\cdots \height(\alpha_k), 
\end{equation}
where
\begin{align*}
C_0&:=3588.1 \frac{(14k)^k(k+1)^{k+1}}{k!}\frac{p}{(f_p\log p)^{k+1}},\\
B&:=\max\{|b_1|, \ldots, |b_k|\}. 
\end{align*}
\end{theorem}

\begin{proof}
We use the Main Theorem from~\cite{Yu13} 
with~$n$ replaced by~$k$. In our special case we have ${d=2}$ and 
$$
q=2, \qquad e_\gerp=1, 
$$
because ${p\ge 7}$. Furthermore, in our case we have
$$
\omega(d)= \frac{1}{2(\log6)^3}, \qquad  q^u=2, 
$$
where $\omega(d)$ is defined in \cite[equation (1.15)]{Yu13}, and  $q^u$ is defined therein in the paragraph between the displayed equations~1.3) and~(1.4).

In our case, parameters ${a^{(1)}, c^{(1)}, a_0^{(1)}, a_1^{(1)}, a_2^{(1)}}$, defined in Section~1.3 of~\cite{Yu13},  have the following values:
\begin{align*}
a^{(1)}&=7\frac{p-1}{p-2}, & c^{(1)}&= 1794, & a_0^{(1)}&=2+\log7, \\ 
a_1^{(1)}&=4.71, & a_2^{(1)}&=4.71+\log2. 
\end{align*}
Substituting all this to the definition of~$C_1$ and $h^{(1)}$ in \cite[equations (1.9) and (1.13)]{Yu13}, we obtain 
\begin{align}
C_1&= 1794 \left(7\frac{p-1}{p-2}\right)^k \frac{k^k(k+1)^{k+1}2^{k+2}}{k!\cdot 2f_p\log p} \nonumber\\
\label{emaxes}
&\hphantom{=}\times\max \left\{\frac{p^{f_p}}{\delta(f_p\log p)^{k+1}}, \frac{e^k}{k^k}\right\} 
\max\{\log(2e^4(k+1)), f_p\log p\}, \\
h^{(1)}&=\max\left\{\log\frac{B'}{2(\log6)^3},\log B^\circ,  (k+1)(k'+\log k'), (k+1)f_p\log p \right\}, 
\end{align}
where 
\begin{align*}
B'&:= \max_{1\le j\le n}\left\{\frac{b_n}{\height(\alpha_j)}+ \frac{b_j}{\height(\alpha_n)}\right\},& 
B^\circ& :=\min\{|b_1|, \ldots, |b_n|\}, \\ 
k'&:= (2+\log7)k+ 4.71+\log2,  
\end{align*}
and~$\delta$ is defined in equation~(1.8) of~\cite{Yu13} (here we use the fact the the Kummer condition~\eqref{ekummer} is satisfied). Note that  
\begin{equation}
\label{edel}
\delta\ge 
\begin{cases}
1 & \text{in any case}\\
p-1, & \text{if ${f_p=2}$}, 
\end{cases}
\end{equation}
as follows from our assumption~\eqref{enormone}.

We may assume, by renumbering,  that 
$$
\nu_p(b_k) \le \nu_p(b_i) \qquad (i=1, \ldots, k-1). 
$$
Then the Main Theorem of~\cite{Yu13} implies that 
$$
\nu_\gerp\bigl(\alpha_1^{b_1}\cdots \alpha_k^{b_k}-1\bigr) \le C_1 h^{(1)}\height(\alpha_1)\cdots \height(\alpha_k). 
$$
To deduce~\eqref{eyu}, we only have to prove that 
$$
C_1\le C_0, \qquad h^{(1)} \le \max\{\log B, (k+1)f_p\log p\}. 
$$
Condition~\eqref{etechnical} implies that that second $\max$ in~\eqref{emaxes} is $f_p\log p$. 
Hence
\begin{equation}
\label{econe}
C_1= 3588 \left(14\frac{p-1}{p-2}\right)^k \frac{k^k(k+1)^{k+1}}{k!} 
\max \left\{\frac{p^{f_p}}{\delta(f_p\log p)^{k+1}}, \frac{e^k}{k^k}\right\}. 
\end{equation}
Using~\eqref{etechnical} and the assumption ${k\ge 8}$,  we find that 
\begin{equation*}
\left(\frac{p-1}{p-2}\right)^k \le \left(1+\frac{1}{e^{3k}k^3-2}\right)^k <e^{1/(64e^{24}-1/4)}. 
\end{equation*}
Next, using~\eqref{etechnical} and~\eqref{edel}, we obtain 
$$
\frac{p^{f_p}}{\delta} \le \frac{p}{p-1}p \le \frac{512e^{24}}{512e^{24}-1} p. 
$$
Finally, let us show that
\begin{equation}
\frac{p}{(f_p\log p)^{k+1}}>\frac{e^k}{k^k}. 
\end{equation}
The function ${x\mapsto x/(\log x)^{k+1}}$ is increasing for ${x\ge e^{k+1}}$. Hence~\eqref{etechnical} implies that 
\begin{align*}
\frac{p}{(f_p\log p)^{k+1}}\left/\frac{e^k}{k^k}\right. & \ge \frac{e^{3k}k^3}{(6k+6\log k)^{k+1}}\left/\frac{e^k}{k^k}\right.\\
& =   \left(\frac{e^2}{6}\right)^k\left(1+\frac{\log k}{k}\right)^{-k}\frac{k^3}{6(k+\log k)}\\
& \ge   \left(\frac{e^2}{6}\right)^k\frac{k^2}{6(k+\log k)} >1, 
\end{align*}
because ${k^2>6(k+\log k)}$ when ${k\ge 8}$. 

Since 
$$
\frac{3588\cdot e^{1/(64e^{24}-1/2)}\cdot 512e^{24}}{{512e^{24}-1}}<3588.1,
$$ 
we have ${C_1<C_0}$. 

Next, it follows from~\eqref{ehegeta} that
${B' \le 4B/\log \eta}$, 
which implies that 
$$
\frac{B'}{2(\log6)^3} < B.
$$ 
We also have trivially ${B^\circ\le B}$. Hence the first two terms in the $\max$ defining $h^{(1)}$ do not exceed $\log B$. 
Finally, using~\eqref{etechnical}, it is easy to prove that ${f_p\log p > k'+\log k'}$. Hence ${h^{(1)} \le \max\{\log B, (k+1)f_p\log p\}}$. The theorem is proved. 
\end{proof}

\section{Primes split in $K$}
\label{sesplit}
A prime number~$q$ splits in~$K$ if and only if ${q\equiv \pm1\pmod 5}$. 
We denote by ${q_2,q_3, q_4, \ldots}$ the sequence of prime numbers split in~$K$ (it will be convenient to start the numbering from~$2$):
$$
q_2=11, \quad q_3=19, \quad q_4=29, \quad q_5=31, \ldots. 
$$
Since~$q_k$ splits in~$K$, it factorizes in~$K$ as ${q_k=\gerq_k\gerq_k^\sigma}$. Since the field~$K$ is of class number~$1$, the fractional ideal $\gerq/\gerq^\sigma$ is principal.  Let ${\theta_k\in K^\times}$ be its generator. Then ${\norm(\theta_k)=\pm1}$. Replacing, if necessary, $\theta_k$ by $\theta_k\eta$, we may assume that 
\begin{equation}
\label{enormonetheta}
\norm(\theta_k)=1. 
\end{equation}
Replacing~$\theta_k$ by $\theta_k\eta^{2m}$ with a suitable integer~$m$, we may assume that 
$$
\eta^{-1}\le |\theta_k|, |\theta_k^\sigma|\le \eta. 
$$
It follows that 
\begin{equation}
\label{ehethek}
\height(\theta_k) \le \frac12\log(q_k\eta) \qquad (k=2,3,\ldots). 
\end{equation}
We also have the ``Kummer property'':
\begin{equation}
\label{ekum}
[K(\sqrt{\gamma}, \sqrt{\theta_2}, \ldots, \sqrt{\theta_k}):K]=2^k \qquad (k=2,3, \ldots). 
\end{equation}
Indeed, we have clearly ${[K(\sqrt\gamma):K]=2}$. Next, for every~$k$, the prime~$\gerq_k$ ramifies in the field $K(\sqrt{\theta_k})$, but not in $K(\sqrt\gamma)$ and $K(\sqrt{\theta_j})$ for ${j\ne k}$. Hence ${\sqrt{\theta_k}\notin K(\sqrt\gamma, \sqrt{\theta_2}, \ldots, \sqrt{\theta_{k-1}})}$, and~\eqref{ekum} follows by induction in~$k$. 

The following lemma will be used in the proof of theorem \ref{thm_bigkappa}.

\begin{lemma}\label{qkupperbtre}
For $k\ge 500000$, $\eta q_k<k^{1.3}$.
\end{lemma}

\begin{proof}
Assume first that $k>80802434$. Then $q_k>3375517771 = q_{80802434}$. By Theorem 1.4 of \cite{BMOR18}, we have
\[
\pi(q_k; 5, a) > \frac{q_k}{4\log q_k}\quad (a=1,2,3,4),
\]
where $\pi(x; m,a)$ counts primes $p\le x$ satisfying $p\equiv a\pmod m$.
Hence \[
k-1 = \pi(q_k; 5, 1) + \pi(q_k; 5, 4) > \frac{q_k}{2\log q_k}.
\]
Since the function $x\mapsto x^{0.3}-\eta (2\log x)^{1.3}$ is increasing when $x>76$, 
this implies that\[
k^{1.3} > \left(
\frac{q_k}{2\log q_k}
\right)^{1.3} >\eta q_k.
\]
It remains to verify the statement for \[
500000\le k \le 80802434.
\]This can be done by  computer algebra system, for example, SageMath\cite{sagemath}. For the script used, see Appendix A2.
\end{proof}

\section{Cyclotomic polynomial and primitive divisors}
\label{sprim}

Let $\Phi_n(X)$ denote the $n^\tho$ cyclotomic polynomial. To start with, let us show that for ${z\in \C}$ and ${r>0}$ satisfying  ${|z|\le r<1}$ we have 
\begin{equation}
\label{ephiup}
\bigl|\log|\Phi_n(z)|\bigr| \le \frac{|\log(1-r)|}{1-r}\frac{|z|}{r}. 
\end{equation}
Indeed, for ${|z|\le r<1}$ the Schwarz Lemma implies that 
$$
\bigl|\log|1+z|\bigr| \le |\log(1+z)| \le \frac{|\log(1-r)|}{r}|z|. 
$$
Hence 
\begin{align*}
\bigl|\log|\Phi_n(z)|\bigr|&= \left|\sum_{m\mid n}\mu\left(\frac nm\right)\log|1-z^m|\right|\\
&\le \frac{|\log(1-r)|}{r}\sum_{m=1}^\infty |z|^m
=  \frac{|\log(1-r)|}{1-|z|}\frac{|z|}{r}, 
\end{align*}
which proves~\eqref{ephiup}. 

Applying~\eqref{ephiup} with ${z=r=\gamma^{-1}}$, we obtain 
$$
\bigl|\log\Phi_n(\gamma^{-1})\bigr|\le  \frac{|\log(1-\gamma^{-1})|}{1-\gamma^{-1}}<0.24. 
$$

Next, if ${\beta\in K^\times}$ then 
$$
\sum_\gerp\nu_\gerp(\beta)\log\norm\gerp = \log|\beta|+\log|\beta^\sigma|, 
$$
the sum being over the primes of~$K$. Applying the above identity with ${\beta=\Phi_n(\gamma)}$, and noting that ${\gamma^\sigma=\gamma^{-1}}$, we obtain 
\begin{align}
\label{esum}
\sum_\gerp\nu_\gerp(\Phi_n(\gamma))\log\norm\gerp &= \log|\Phi_n(\gamma)|+\log|\Phi_n(\gamma^{-1})|\\
& =\ph(n)\log|\gamma|+2\log|\Phi_n(\gamma^{-1})|\nonumber\\ 
&= \ph(n)\log|\gamma|+O_1(0.48), \nonumber
\end{align}
where ${O_1(\cdot)}$ is the quantitative version of the familiar $O(\cdot)$ notation: ${A=O_1(B)}$ means that ${|A|\le B}$.

Call a $K$-prime~$\gerp$ primitive divisor of ${\gamma^n-1}$ if 
$$
\gerp\nmid 5, \qquad \nu_\gerp(\gamma^n-1) \ge 1, \qquad \nu_\gerp(\gamma^m-1)=0 \quad (m=1, \ldots, n-1). 
$$
Clearly,~$\gerp$ is a primitive divisor of ${\gamma^n-1}$ if and only if  the underlying rational prime~$p$ is  a primitive divisor of~$F_n$. 

Let~$\gerp$ be not a primitive divisor of ${\gamma^n-1}$. A famous result of Schinzel \cite[Lemma~4]{Sc74} (see also \cite[Lemma~4.5]{BL21}) implies that ${\nu_\gerp(\Phi_n(\gamma)) \le \nu_p(n)}$ when ${n\ne 6}$. Restricting the sum in~\eqref{esum} to primitive divisors of ${\gamma^n-1}$,  we obtain  
\begin{equation*}
\sum_{\text{$\gerp$ primitive}}\nu_\gerp(\Phi_n(\gamma))\log\norm\gerp \ge \ph(n)\log|\gamma|-2\log n-0.48 \qquad (n\ne 6). 
\end{equation*}
We want to make this more convenient to use. Theorem~15 of~\cite{RS62} implies that, for ${n\ge e^{100}}$, we have 
$$
\ph(n) \ge 0.526\frac{n}{\log\log n} . 
$$
Hence
\begin{equation*}
\sum_{\text{$\gerp$ primitive}}\nu_\gerp(\Phi_n(\gamma))\log\norm\gerp \ge  0.5\frac{n}{\log\log n} \qquad (n\ge e^{100}). 
\end{equation*}
Now let us  express this in terms of the primitive divisors of~$F_n$ rather than ${\gamma^n-1}$. For any prime~$\gerp$ of~$K$ we have ${\nu_\gerp(F_n)=\nu_\gerp(\gamma^n-1)}$. Hence 
\begin{align*}
\sum_{\text{$p$ primitive}}\nu_p(F_n)\log p
&=\frac12\sum_{\text{$\gerp$ primitive}}\nu_\gerp(F_n)\log\norm\gerp \\
&=\frac12\sum_{\text{$\gerp$ primitive}}\nu_\gerp(\gamma^n-1)\log\norm\gerp \\
&\ge \frac12 \sum_{\text{$\gerp$ primitive}}\nu_\gerp(\Phi_n(\gamma))\log\norm\gerp, 
\end{align*}
which implies that
\begin{equation}
\label{eliou}
\sum_{\text{$p$ primitive}}\nu_p(F_n)\log p \ge 0.25\frac{n}{\log\log n} \qquad (n\ge e^{100}). 
\end{equation}

\section{Stewart's argument}

Stewart's argument gives a non-trivial lower estimate for ${\nu_\gerp(\gamma^n-1)}$. Let ${k\ge 2}$ be an integer, and ${q_2, \ldots, q_k}$ the first ${k-1}$ primes split in~$K$, see Section~\ref{sesplit}.  Define
\begin{equation}
\label{elaka}
\Theta_k:= \bigl((k+1)\log\eta+\log(q_2\cdots q_k)\bigr)\prod_{j=2}^k\log(\eta q_j). 
\end{equation}

\begin{proposition}
\label{prorder}
Let~$n$ be a positive integer and~$p$ is a primitive divisor of~$F_n$. 
Let ${k\ge 8}$ be an integer such that
\begin{equation}
\label{etechnicalbis}
p\ge e^{3k}k^3. 
\end{equation}
Then
\begin{equation*}
\nu_p(F_n) \le 3588.1 (7k)^k\frac{(k+1)^{k+2}}{k!}\frac{p}{(f_p\log p)^{k}}\Theta_k. 
\end{equation*}

\end{proposition}

\begin{proof}
Let ${q_2, \ldots, q_k}$ and ${\theta_2, \ldots, \theta_k}$ as in Section~\ref{sesplit}. We define 
$$
\theta_1: =\frac{\gamma}{\theta_2\cdots\theta_k}. 
$$
We deduce from~\eqref{ehega} and~\eqref{ehethek} that  
$$
\height(\theta_1) \le \height(\gamma) + \height(\theta_2)+\cdots+\height(\theta_k) \le \frac12\bigl((k+1)\log\eta+\log(q_2\cdots q_k)\bigr).  
$$
Hence 
$$
\height(\theta_1)\cdots \height(\theta_k)\le \frac{\Theta_k}{2^k}. 
$$

Let~$\gerp$ be a prime of~$K$ dividing~$p$. Then ${\nu_p(F_n)=\nu_\gerp(\gamma^n-1)}$. 
We have 
${\gamma^n-1= \theta_1^n\theta_2^n\cdots \theta_k^n-1}$.  
We want to apply Theorem~\ref{thyu} with ${\alpha_j=\theta_j}$ and with ${b_1=\cdots b_k =n}$. The Kummer condition~\eqref{ekum} and the norm~$1$ condition~\eqref{enormone} are satisfied by~\eqref{ekum} and~\eqref{enormonetheta}; also,~\eqref{etechnical} is satisfied by~\eqref{etechnicalbis}. 

In~\eqref{eyu} we may replace the  maximum by ${(k+1)f_p\log p}$; indeed, we have ${p\ge n-1}$ because~$p$ is a primitive divisor, and 
$$
(k+1)f_p\log p>\log n=\log B.
$$
We obtain 
\begin{align*}
\nu_p(F_n)&=\nu_\gerp(\gamma^n-1)\\
&=\nu_\gerp(\theta_1^n\theta_2^n\cdots \theta_k^n-1)\\
&\le 3588.1 \frac{(14k)^k(k+1)^{k+1}}{k!}\frac{p}{(f_p\log p)^{k+1}}(k+1)f_p\log p\cdot\height(\theta_1)\cdots \height(\theta_k)\\
&\le 3588.1 (7k)^k\frac{(k+1)^{k+2}}{k!}\frac{p}{(f_p\log p)^{k}}\Theta_k, 
\end{align*}
as wanted. 
\end{proof}

\section{Proof of Theorem \ref{thm_smallkappa}}
Let~$\kappa$ and~$k$ be positive integers, ${k\ge 8}$. 

In this section we will assume that~$F_n$ has no primitive prime divisor distinct from one of the numbers 
\begin{equation*}
jn\pm1 \qquad (j=1, \ldots, \kappa), 
\end{equation*}
and that~$n$ satisfies 
\begin{equation}
\label{enefour}
n> \max\{e^{3k}k^3+1, e^{100}\}, 
\end{equation}
We are going to  show that, under these assumptions, the upper estimate
\begin{equation}
\label{emain}
\log n < \Lambda(k, \kappa) (2\log\Lambda (k, \kappa))^{1/(k-1)} 
\end{equation}
holds, where
$$
\Lambda(k, \kappa):= 
7k(k+1)\left(\left(1+\left(\frac12\right)^k\right)50233.5\frac{(k+1)^3}{(k-1)!} \kappa(\kappa+1)\Theta_k\right)^{1/(k-1)}.  
$$
and~$\Theta_k$ is defined in~\eqref{elaka}.

In particular, specifying ${\kappa=1}$ and ${k=22}$,
this would imply that, when all primitive divisors of~$n$ are among $n+1$ and $n-1$, we have either 
$$
n\le \max\left\{10648\mathrm{e}^{66}+1, \mathrm{e}^{100}\right\}=\mathrm{e}^{100}, 
$$
or
$$
\log n < 7606.3.
$$

Since ${n>e^{100}}$ by~\eqref{enefour}, we may use estimate~\eqref{eliou}:
\begin{equation}
\label{elioubis}
\sum_{\text{$p$ primitive}}\nu_p(F_n)\log p \ge 0.25\frac{n}{\log\log n} . 
\end{equation}
Now let us bound the sum on the left from above.

As noted in the introduction, a primitive divisor~$p$ of~$F_n$ satisfies the congruence ${p\equiv \pm1\pmod n}$. 
More precisely, a primitive divisor~$p$ satisfies
$$
p\equiv 
\begin{cases}
\hphantom{-}1\pmod n &\text{if~$p$ splits in~$K$}, \\
-1\pmod n &\text{if~$p$ is inert in~$K$}. 
\end{cases}
$$
According to this, we split the sum in~\eqref{elioubis} into two parts, one collecting the terms with ${p\equiv 1\pmod n}$ and the other those with ${p\equiv -1\pmod n}$; in symbols: 
\begin{align*}
\sum_{\text{$p$ primitive}}\nu_p(F_n)\log p 
&=\Sigma_++\Sigma_-,\\
\Sigma_+&:=\sum_{\genfrac{}{}{0pt}{}{\text{$p$ primitive}}{p\equiv1\pmod n}}\nu_p(F_n)\log p, \\
\Sigma_-&:=\sum_{\genfrac{}{}{0pt}{}{\text{$p$ primitive}}{p\equiv-1\pmod n}}\nu_p(F_n)\log p. 
\end{align*} 
Now let~$p$ be a primitive divisor of~$F_n$. Since ${p\ge n-1\ge e^{3k}k^3}$ by~\eqref{enefour}, we may apply Proposition~\ref{prorder}. We obtain 
$$
\nu_p(F_n)\log p \le  3588.1\left(\frac{1}{f_p}\right)^k (7k)^k\frac{(k+1)^{k+2}}{k!}\frac{p}{(\log p)^{k-1}}\Theta_k. 
$$
If ${p=jn+1}$  then, using ${n\ge e^{100}}$, we obtain  
$$
\frac{p}{(\log p)^{k-1}} \le j\frac{n+1/j}{(\log n)^{k-1}}\le\left(1+\frac{1}{e^{100}}\right) j\frac{n}{(\log n)^{k-1}}. 
$$
Since ${f_p=1}$ for ${p=jn+1}$, this gives the following estimate for $\Sigma_+$:
\begin{align}
\Sigma_+ &\le \left(\sum_{j=1}^\kappa j\right) \left(1+\frac{1}{e^{100}}\right) 3588.1   \cdot(7k)^{k-1}\frac{7(k+1)^{k+2}}{(k-1)!} \frac{n}{(\log n)^{k-1}} \Theta_k\nonumber \\
\label{esigpl}
&\le 12558.35(7k)^{k-1}\frac{(k+1)^{k+2}}{(k-1)!}\kappa(\kappa+1)\frac{n}{(\log n)^{k-1}} \Theta_k. 
\end{align}
Similarly, if ${p=jn-1}$  then  
$$
\frac{p}{(\log p)^{k-1}} \le j\frac{n-1/j}{(\log (n-1/j))^{k-1}}\le j\frac{n}{(\log n)^{k-1}}, 
$$
because the function ${x\mapsto x/(\log x)^{k-1}}$ is increasing for ${x\ge e^{k-1}}$. 
Since ${f_p=2}$ for ${p=jn-1}$, we estimate $\Sigma_-$ as 
\begin{equation}
\label{esigmin}
\Sigma_- \le \left(\frac12\right)^k12558.35(7k)^{k-1}\frac{(k+1)^{k+2}}{(k-1)!}\kappa(\kappa+1)\frac{n}{(\log n)^{k-1}} \Theta_k. 
\end{equation}
Combing estimates~\eqref{elioubis},~\eqref{esigpl} and~\eqref{esigmin}, we obtain 
$$
\log n \le \Lambda(k,\kappa) (\log\log n)^{1/(k-1)}. 
$$
Applying to this Lemma~\ref{lna}, we deduce~\eqref{emain}.

\section{Proof of Theorem \ref{thm_bigkappa}} 
Let $\kappa$ and $k$ be positive integers, $\kappa\ge \exp(250000)$, $k\ge 500000$.

We assume that~$F_n$ has no primitive prime divisor distinct from one of the numbers 
\begin{equation*}
jn\pm1 \qquad (j=1, \ldots, \kappa), 
\end{equation*}
and that~$n$ satisfies 
\begin{equation}
\label{enefour18}
n> \max\{e^{3k}k^3+1, e^{100}\}.
\end{equation}

As in the proof of theorem  \ref{thm_smallkappa}, the upper estimate
\begin{equation}
\label{emain18}
\log n \le \Lambda(k, \kappa) (2\log\Lambda (k, \kappa))^{1/(k-1)} 
\end{equation}
holds, where
$$
\Lambda(k, \kappa)= 
7k(k+1)\left(\left(1+\left(\frac12\right)^k\right)50233.5\frac{(k+1)^3}{(k-1)!} \kappa(\kappa+1)\Theta_k\right)^{1/(k-1)}.  
$$
and~$\Theta_k$ is defined in~\eqref{elaka}.

By lemma \ref{qkupperbtre}, we estimate $\Theta_k$ as \begin{align*}
\Theta_k & =\left(2\log\eta + \sum_{i=2}^k\log(\eta q_i)\right)\prod_{i=2}^k \log(\eta q_i)\\
&< \left(
2\log\eta +1.3(k-1)\log k
\right) (1.3\log k)^{k-1}.
\end{align*}
Hence \begin{align*}\label{xcluppb}
\Lambda(k, \kappa) & =7k(k+1)\left(\left(1+\left(\frac12\right)^k\right)50233.5\frac{(k+1)^3}{(k-1)!} \kappa(\kappa+1)\Theta_k\right)^{1/(k-1)}\\
& <9.1k\log k \cdot C(k)\cdot M^{1/(k-1)},
\end{align*}
where \[
C(k ) := \left(\left(
1+\left(
\frac{1}{2}
\right)^k
\right)50233.5\frac{(k+1)^{k+2}}{(k-1)!} (2\log\eta +1.3(k-1)\log k)
\right)^{1/(k-1)},
\]\[
M: = \kappa(\kappa+1).
\]
Note that when $k\ge 500000$, $C(k)\le C(500000)<2.72$. So \begin{equation}\label{xclupb}
\Lambda(k,\kappa)<24.8k\log k M^{1/(k-1)}.
\end{equation}

Since \eqref{emain18} holds for any $k\ge 500000$,
we may take \begin{equation}\label{kMsetting}
k= \lfloor \log(M)\rfloor.
\end{equation}
Then \begin{equation}\label{logxclupb}
\log\Lambda(k,\kappa)<\log(24.8)+\log k+\log\log k+\frac{1}{k-1}\log M<1.52\log k.
\end{equation}

Combining \eqref{emain18}\eqref{xclupb}\eqref{kMsetting}\eqref{logxclupb}, we obtain 
\begin{equation}\label{finalassm}
\begin{aligned}
\log n & \le \Lambda(k,\kappa)(2\log \Lambda(k,\kappa))^{1/(k-1)}\\
& < 24.8k\log kM^{1/(k-1)}(3.04\log k)^{1/(k-1)}\\
&< 67.42\log M\log\log M.
\end{aligned}
\end{equation}
Indeed, 
\[
M^{1/(k-1)}=\mathrm{e}^{\log M/(k-1)}\le \mathrm{e}^{1+2/(\log M-2)}\le \mathrm{e}^{1+2/(500000-2)}
\]
and \[
(3.04\log k)^{1/(k-1)}= \exp \left(
\frac{\log(3.04\log\log M)}{\log M-2}
\right) \le \exp \left(
\frac{\log(3.04\log 500000)}{500000-2}
\right).
\]

Moreover, we can also estimate $\log M$ and $\log\log M$ as \[
\log M=\log\kappa+\log(\kappa+1)\le 2.0001\log\kappa
\] 
and \[
\log\log M\le 1.056\log\log\kappa.
\]
Thus \eqref{finalassm} reduces to \[
\log n<143\log\kappa\log\log\kappa.
\]
This completes the proof.

\paragraph{Acknowledgments} The author is very grateful to Yuri Bilu, for his many substantial revisions of this paper and for very helpful discussions. The author also acknowledges support of China Scholarship Council
Grant CSC202008310189.

{\footnotesize

\bibliographystyle{amsplain}
\bibliography{bigPFn}

\paragraph{Haojie Hong} 
~\\
Institut de Mathématiques de Bordeaux, Université de Bordeaux \& CNRS, Talence, France\\
Address: A33, 351, cours de la Libération, 33405, Talence, France\\
E-mail address: haojie.hong@math.u-bordeaux.fr

}

\newpage

\section*{Appendix: SageMath scripts}
{\footnotesize

\subsection*{A1. For computing $n_0$ for $1\le \kappa\le 10^6$}
\begin{verbatim}
def calculate_alpha():
    alpha = (1 + sqrt(5)) / 2
    return alpha.n()

def find_primes(X):
    l = []
    p = next_prime(X)
    while p > 3:
        p = previous_prime(p)
        l.append(p)
    return l

def filter_primes(l):
    o = []
    for i in l:
        if mod(i, 5) == 1 or mod(i, 5) == -1:
            o.append(i)
    return o

def calculate_thetak(o, alpha):
    odelmax = o[1:]
    logsum = 0
    logprod = 1
    for p in odelmax:
        logsum += log(alpha)+ log(p)
        logprod *= log(alpha * p)
    thetak = (2 * log(alpha) + logsum) * logprod
    return thetak
        
def calculate_Lambdak(k, ka, thetak):
    Lambdak = 7  * k *((1 + 0.5^k) * 50233.5  *  (k+1)^(k+2) 
    * ka * (ka+1) * thetak/ factorial(k-1))^(1/(k-1))
    return Lambdak

def calculate_dsn(Lambdak, k):
    dsn = Lambdak * (2 * log(Lambdak))^(1/(k-1))
    return dsn.n()

def calculate_all_Lambdak(val):
    results = []
    for X in range(100, 1000):
        alpha = calculate_alpha()
        primes = find_primes(X)
        filtered_primes = filter_primes(primes)
        k = len(filtered_primes)
        ka = val    
        thetak = calculate_thetak(filtered_primes, alpha)
        Lambdak = calculate_Lambdak(k, ka, thetak)
        dsn = calculate_dsn(Lambdak, k)
        results.append((dsn, k, Lambdak))
    min_dsn, min_k, min_Lambdak = min(results, key=lambda x: x[0])
    print(f"For k={min_k},min_logn={min_dsn}")

calculate_all_Lambdak(1)     # change the values of $\kappa$ here

\end{verbatim}

\subsection*{A2. For lemma \ref{qkupperbtre}}
\begin{verbatim}
def splitq(k):  # calculates q_k
    n = 1
    for q in Primes():
        if q % 5 == 1 or q % 5 == 4:
            n += 1
        if n == k:
            return q

def piksplit(x):  # calculates the biggest k such that q_k<=x
    k = 1
    for q in Primes():
        if q > x:
            break
        if q % 5 == 1 or q % 5 == 4:
            k += 1
    return k

def pikest(x, y, a):  # verifies eta q_k <= k^a for all k
 satisfying y<=q_k<=x
    k = 1
    eta = (1 + sqrt(5)) / 2
    for q in Primes():
        if q < y:
            continue
        if q > x:
            break
        if q % 5 == 1 or q % 5 == 4:
            k += 1
        est = k**a / eta
        if q > est:
            print(k)
            print(q)
            return
    print("confirmed")
    
    pikest(500000, 80802434, 1.3)
\end{verbatim}
}

\end{document}